\documentclass[12pt,a4paper,oneside]{amsart}
\usepackage[scale=0.7]{geometry}
\usepackage[colorlinks=true,citecolor=blue]{hyperref}
\newtheorem{theorem}{Theorem}[section]
\newtheorem{lemma}[theorem]{Lemma}

\newtheorem{prop}[theorem]{Proposition}
\usepackage{amsfonts, amsmath, amssymb, amsthm, hyperref}
\usepackage{parskip}
\usepackage{tikz}

\usepackage{txfonts}
\usepackage[T1]{fontenc}

\newcommand{\mb}[1]{\mathbb{#1}}
\DeclareMathOperator{\conv}{conv}
\usepackage{verbatim}
\usepackage{tikz}
\usepackage{courier}
\usepackage{units}
\usepackage{lineno}
\usepackage{booktabs}

\newtheorem{corollary}[theorem]{Corollary}

\newcommand{\cal}{\mathcal}

\title[Combinatorial Colorful Carath\'{e}odory]{A combinatorial
  version of the colorful Carath\'{e}odory theorem}  
\author{Andreas F. Holmsen} 
\thanks{Supported by Basic Science Research Program through the
  National Research Foundation of Korea (NRF) funded by the Ministry
  of Education, Science and Technology (NRF-2010-0021048).}  

\address{Department of Mathematical Sciences, KAIST\\ 291 Daehak-ro, 
Daejeon 305-701\\ South Korea \\ Tel/Fax: +82 42-350-7300/ +82 42-350-2710}

\begin{document}

\begin{abstract} We give the following extension of B\'{a}r\'{a}ny's
  colorful Carath\'{e}odory theorem:   Let $\cal M$ be an oriented
  matroid, $\cal 
  N$ a matroid with rank function $\rho$, both defined on the same 
  ground set $V$ and satisfying $rk ({\cal M}) < rk ({\cal N})$. If every
  $A\subset V$ with $\rho(V - A) < rk ({\cal M})$ contains a positive
  circuit of $\cal{M}$, then there is a positive circuit of $\cal
  M$ which is independent in $\cal N$.  
\end{abstract}

\maketitle

\section{Introduction}

One of the cornerstones of convexity is Carath\'{e}odory's theorem which
states that, given a set $P\subset \mb{R}^d$ and a point $x$ in the
convex hull of $P$, i.e. $x\in \conv P$, there exists a subset
$Q\subset P$  such that $|Q|\leq d+1$ and $x\in \conv Q$. In 1982,
B\'{a}r\'{a}ny \cite{colorcar} gave the following generalization of
Carath\'{e}odory's theorem. 

\begin{theorem}[Colorful Carath\'{e}odory]\label{baraCC}
Let $P_1, \dots, P_{d+1}$ be point sets in $\mb{R}^d$. If $x \in
\bigcap_{i=1}^{d+1}\conv P_i $, then there exists $p_1 \in P_1,
\cdots,  p_{d+1}\in P_{d+1}$ such that $x \in \conv \{p_1, \dots,
p_{d+1}\}$.  
\end{theorem}

The name originates from thinking of the $P_i$ as distinct
color classes. The conclusion tells us the point $x$ is
contained in a ``colorful simplex'', that is a simplex whose vertices
are of all distinct colors. Notice that B\'{a}r\'{a}ny's theorem reduces to
Carath\'{e}odory's theorem when the $P_i$ are equal. 

Theorem \ref{baraCC} has many applications in discrete geometry
\cite{mato}, and gives rise to interesting variations of
linear programming \cite{collin}.  
It is easily seen that the hypothesis of Theorem \ref{baraCC} is
not a necessary condition, and the following weakening of the hypothesis was
recently discovered \cite{very, strong}.   

\begin{theorem}[Strong Colorful Carath\'{e}odory] \label{strongCC}
Let $P_1, \dots, P_{d+1}$ be non-empty point sets in $\mb{R}^d$. If
$x \in \bigcap_{1\leq i<j}^{d+1}\conv (P_i \cup P_j)$, then 
there exists $p_1 \in P_1 , \dots, p_{d+1} \in P_{d+1}$ 
 such that $x \in \conv \{p_1, \dots,
p_{d+1}\}$. 
\end{theorem}

For applications of the strengthened version, see \cite{very}.
The goal of this paper is to give a twofold generalization of Theorem
\ref{strongCC}.   

\begin{enumerate}
\item Points in $\mb{R}^d$ will be replaced by an {\em oriented
    matroid}. Every vector configuration in $\mb{R}^d$ gives rise to
  an oriented matroid, but the converse does not hold. In fact, there
  are far more {\em non-realizable} oriented matroids than realizable
  ones. The question of whether the Colorful Carath\'{e}odory theorems extend to
  oriented matroids is a natural one, and has been asked ever since
  B\'{a}r\'{a}ny first introduced his result. (The rank 3 case was
  recently considered in \cite{rankthree}.) 

\item The color classes will be replaced by a {\em matroid}. It was
  noticed by Kalai and Meshulam \cite{KandM} that the colorful
  simplices are actually playing the role of bases of the {\em
    transversal matroid} of the family $\{P_1, \dots, P_{d+1}\}$. In the closely 
  related setting of {\em $d$-Leray complexes}, they showed that the
  transversal matroid can be replaced by an arbitrary matroid. Our
  main result (Theorem \ref{matroidCC}) can be thought of as dual to
  the result of Kalai and Meshulam, and our proof is a
  modification of theirs. 
\end{enumerate}

\subsection{Matroids and oriented matroids}
A {\em matroid} is a combinatorial structure designed to capture the
notion of linear independence in vector spaces. There are numerous
``cryptomorphic'' axiom systems which define a
matroid. For an introduction to matroid theory, see
e.g. \cite{oxley}. The independent sets of a matroid naturally give
rise to a simplicial complex whose topology will be of importance to
us. For more
information in this direction we refer the reader to \cite{bjorn}. For
a matroid $\cal N$ let $rk ({\cal N})$ denote its rank. 

An {\em oriented matroid} can be thought of as a combinatorial
abstraction of a finite vector configuration spanning a vector space
over an {\em ordered field}. As for 
ordinary matroids, there are several cryptomorphic axiom systems which
define an oriented matroid. Of importance to us are the {\em positive
  circuits}. In the analogy of vector spaces, these correspond to
the {\em minimal positive linear dependencies} of the
configuration. For details, we refer the
reader to \cite{OMS, gebertHandbook}. For an oriented matroid $\cal M$
let $rk (\cal M)$ denote the rank of its underlying matroid. 

In this paper all matroids and oriented matroids are considered to be
loopless.

\subsection{Main result}

\begin{theorem} \label{matroidCC}
  Let $\cal M$ be an oriented matroid, $\cal
  N$ a matroid with rank function $\rho$, both defined on the same 
  ground set $V$ and satisfying $rk ({\cal M}) < rk ({\cal N})$. If every
  $A\subset V$ with $\rho(V - A) < rk ({\cal M})$ contains a positive
  circuit of $\cal{M}$, then there is a positive circuit of $\cal M$
  which is independent in $\cal N$.
\end{theorem}

As an immediate corollary we obtain Theorem \ref{strongCC}. To see
this let $\cal M$ be the oriented matroid of the vector configuration
$V = \{ p-x \; |\; p\in P_1\cup \cdots \cup P_{d+1} \}$ and $\cal N$
the transversal matroid of the family $\{P_1, \dots, P_{d+1}\}$.

Our proof  of Theorem \ref{matroidCC} uses topological methods: The
Folkman-Lawrence representation theorem \cite{folkman} allows us to
represent an oriented matroid as an arrangement of open
pseudohemispheres with nice intersection properties. We may therefore
pass to the nerve complex of the arrangement whose homology can be
determined using the Nerve theorem (see \cite{bjorner}). The rest of
the proof (more 
or less) follows the arguments of Kalai and Meshulam
\cite{KandM}.

\section{Preliminaries}
Here we collect the basic facts needed for the proof of Theorem
\ref{matroidCC}. 
\subsection{Simplicial complexes and homology}
First we review some standard notions from simplicial homology.
Let $X$ be a
simplicial complex on $V$. For $W\subset V$ let \[X[W] = 
\{T \in X \: : \: T\subset W\}\] denote the {\em induced subcomplex} on $W$.
For a simplex $S\in X$ let \[lk (S,X) = \{T \in X \; : \; T\cap S =
\emptyset, T \cup S \in X\}\] denote the {\em link} of $S$. 
Let
$\widetilde{H}_j(X)$ denote the $j$-th {\em reduced homology group} of
$X$ with rational coefficients,
and define
\[\eta(X) = \min \{j \: : \: \widetilde{H}_j(X) \neq 0 \} + 1\]
For simplicial complexes $X$ and $Y$ on disjoint vertex sets, the {\em
  join} $X * Y$ is the simplicial complex on the union of their
vertex sets defined as 
\[X * Y = \{S \cup T \: : \: S\in X, T\in Y\}\]
By the K\"{u}nneth formula for the join of simplicial complexes we
obtain the following.  

\begin{corollary} \label{kunneth}
    $\eta(X * Y) = \eta(X)  + \eta(Y)$
        \end{corollary}

 Let $X$ be a simplicial complex on $V$ and suppose $V\notin X$. The
 {\em Alexander dual} $X^\star$ is the simplicial complex on
 $V$ defined as \[X^\star = \{ T \subset V \: :\:
  V- T \notin X\}\]
The homology of $X$ and $X^\star$ are related by
Alexander duality, which says that if $V\notin X$ then $
\widetilde{H}_i(X^\star) \cong 
  \widetilde{H}_{|V|-i-3}(X)$ for all $-1\leq i \leq |V|-2$. 

  \begin{corollary}\label{alex}
    If $V\notin X$, $S\notin X^\star$, and $T = V - S$, then 
\[ \widetilde{H}_i(X^\star[S]) \cong \widetilde{H}_{|S|-i-3}(lk(T,X)) \] 
  \end{corollary}

\subsection{The Nerve theorem}

Let ${\cal F} = \{S_v\}_{v\in V}$ be family of sets. The {\em nerve}
$N_{\cal F}$ is the abstract simplicial complex on $V$
whose simplices consists of those $T\subset V$ such that
$\bigcap_{v\in T}S_v \neq\emptyset$. We recall the following version
of the Nerve theorem. (For a proof see e.g. \cite{bjorner}) 

\begin{theorem}\label{nerves}
  Let ${\cal F} =\{S_v\}_{v\in V}$ be a family of open contractible subsets
  of $\mb{R}^n$ such that every non-empty intersection $\bigcap_{v\in
    W}S_v$ is contractible. Then $\cal F$ and $N_{\cal F}$ are
  homotopy equivalent.
\end{theorem}

\subsection{The independence complex of a matroid}
Let $\cal N$ be a matroid on $V$. The {\em independence complex} of
$\cal N$ is the simplicial complex $Y_{\cal N}$ on $V$ whose simplices are the
independent sets of $\cal N$. In other words, 
\[Y_{\cal N} = \{S\subset V \: : \: S \text{ independent in } \cal N\}\]

The following is a well-known fact. (See e.g. \cite{bjorn, BjKoLo}.)

\begin{lemma}\label{mathomol}
Let $\cal N$ be a matroid on ground set $V$ with rank function $\rho$
and $Y = Y_{\cal N}$ its independence complex. Then $\eta(Y[S]) \geq
\rho (S)$ for 
every non-empty $S\subset V$. 
\end{lemma}

\subsection{Topological representation of oriented matroids}

Let $\cal M$ be an oriented matroid of rank $d$ on the
ground set $V$. (Oriented matroids are assumed to be loopless.) The
Folkman-Lawrence representation theorem 
\cite{folkman} states that $\cal M$ can be represented as an
arrangement $\{S_v\}_{v\in V}$ of {\em oriented pseudospheres} in
$\mb{S}^{d-1}$. Such an arrangement decomposes $\mb{S}^{d-1}$ into a
regular cell complex whose combinatorial structure encodes the
oriented matroid. (See \cite{OMS} for precise definitions and proofs.)   
Equivalently, $\cal M$ can be represented by the collection of
 {\em open pseudohemispheres} $\{h_v\}_{v\in V}$ in
$\mb{S}^{d-1}$, which have the $S_v$ as boundaries and lie on the
``positive'' sides. The crucial fact for us is the following.

\begin{corollary} \label{hemis}
Let $\cal M$ be an oriented matroid on ground set $V$ and
$\{h_v\}_{v\in V}$ a topological representation by pseudohemispheres.
 The intersection $h_W =
  \bigcap_{w\in W} h_w$ is empty or contractible for every $W\subset
  V$. Moreover, $h_W$ is empty if and only if $W$ contains 
  a positive circuit of $\cal M$.
\end{corollary}

\subsection{The support complex of an oriented matroid}

Let $\cal M$ be an oriented matroid on $V$. The {\em support complex}
of $\cal M$ is the simplicial complex $X_{\cal M}$ on $V$ whose
simplices are the subsets of $V$ which do not contain positive
circuits of $\cal M$. That is, \[X_{\cal M} = \{S \subset V \: : \: S
\text{ contains no positive circuit of } \cal M\}\]

\begin{prop}
    \label{homol}
  Let $\cal M$ be an oriented matroid of rank $r$ and $X = X_{\cal M}$
  its support complex. The following hold.
  \begin{enumerate}
  \item $\widetilde{H}_j(X) = 0$ for all $j\geq r$.
    \item $\widetilde{H}_j(lk (S,X)) = 0$ for all $j\geq r-1$ and
      non-empty $S\in X$.
  \end{enumerate}
  \end{prop}

\begin{proof}
Consider a topological representation of $\cal M$ by an arrangement of
pseudohemispheres ${\cal A} = \{h_v\}_{v\in V}$ in $\mb{S}^{r-1}$. 
Corollary \ref{hemis} implies that $X_{\cal M}$ is the nerve
of ${\cal A}$  and that every non-empty intersection $\bigcap_{v\in
  S}h_v$ is contractible. So by the Nerve theorem  $X$
is homotopic to $\bigcup_{v\in V}h_v \subset \mb{S}^{r-1}$, hence
$\widetilde{H}_j(X) = 0$ for all $j\geq r$.   

For the second part, let $\emptyset \neq S\in X$ and let $h_S = \bigcap_{v\in
  S}h_v$. The simplices of $lk(S,X)$ correspond
to subsets $T\subset V\setminus S$ such that the intersection
$\bigcap_{v\in T} h_v \cap h_S$ is non-empty. It follows that $lk(S,X)$ is
the nerve of the family $\{h_v \cap h_S\}_{v\in V\setminus S}$,
and the Nerve theorem implies that $lk(S,X)$ is homotopic
to $\bigcup_{v\in V\setminus S}(h_v\cap h_S) \subset h_S$. Since $h_S$
is homeomorphic to $\mb{R}^{r-1}$ it follows that $\widetilde{H}_j(lk
(S,X)) = 0$ for all $j\geq r-1$. 
\end{proof}

\subsection{Colorful simplices}
Let $Z$ be a simplicial complex on $V$ and $\bigcup_{i=1}^m V_i$ a
partition of $V$. A {\em colorful simplex} of $Z$ is a simplex $S\in
Z$ such that $|S \cap V_i|= 1$ for all $1\leq i \leq m$. Meshulam
\cite{meshulam} gave the following sufficient condition for a
simplicial complex on $\bigcup_{i=1}^m V_i$ to contain a colorful
simplex. (For a short proof based on the Nerve theorem, see \cite{KandM})

\begin{prop} \label{meshulam}
If for all $\emptyset \neq I\subset [m]$ we have
\[ \eta \left( Z\left[ \cup_{i\in I} V_i\right]  \right) \geq |I|\]
then $Z$ contains a colorful simplex.
\end{prop}

\section{Proof of Theorem \ref{matroidCC}}

Let $V = \{v_1, v_2, \dots, v_m\}$, $r = rk (\cal M)$, $X = X_{\cal
  M}$ the support complex of $\cal M$, and $Y = Y_{\cal N}$ the
independence complex of $\cal N$. Make a disjoint copy $V' = \{v_1',
v_2', \dots, v_m'\}$ of $V$ and let $Y'$ be an isomorphic copy of $Y$ on
$V'$. Consider the join $Z = X^\star * Y'$ and let $V_i = \{v_i,
v_i'\}$ for $1\leq i \leq m$. 

Notice that a colorful simplex $S \cup T' \in Z$ implies that $T = V -
S$ is independent in $\cal N$. It also implies $T\notin X$ and therefore
$T$ contains a positive circuit of $\cal M$. The strategy is therefore to
apply Proposition \ref{meshulam} to show that $Z$ contains a colorful
simplex. 

For $\emptyset \neq I \subset [m]$ set $S= \{v_i \: : \: i\in I\}$ and
$S' = \{v_i' \: : \: i\in I\}$. By  Corollary \ref{kunneth} and Lemma
\ref{mathomol} we have

\begin{eqnarray*}
  \eta\left( Z\left[ \cup_{i\in I}V_i \right]\right)
  &=&\eta(X^\star[S] * Y'[S'] ) \\ 
 & = &  \eta(X^\star[S]) + \eta(Y[S]) \\
 & \geq & \eta(X^\star[S]) + \rho(S)  
\end{eqnarray*}

If $S\in X^\star$ then $X^\star[S]$ is contractible, which implies
$\eta(X^\star[S]) = \infty > |I|$. We may therefore assume $S\notin
X^\star$ and consequently $T = V - S \in X$. By hypothesis $\cal M$
contains positive circuits, hence $V\notin X$, so by Corollary \ref{alex}
we have \[\widetilde{H}_i(X^\star[S]) \cong
\widetilde{H}_{|S|-i-3}(lk(T,X)) \]
There are two cases to consider. 

\begin{enumerate}

\item If $S = V$ then $X^\star[S] = X^\star$, $T=\emptyset$, and
  $lk(T,X) = X$. The first case of Lemma \ref{homol} implies that
  $\widetilde{H}_i(X^\star) = 0$ for all $i\leq |S|-r-3$, hence
\[\eta(X^\star) \geq |S|-r-1\]
By hypothesis $\rho(V) = rk({\cal N}) > r$, which implies 
\begin{eqnarray*}
  \eta(Z) &\geq &\eta(X^\star) + \rho(V) \\
          &\geq &(|S| - r- 1) + (r +1) =|V|
\end{eqnarray*}

\item If $S$ is a proper subset of $V$, then the second case of Lemma
  \ref{homol} implies that $\widetilde{H}_i(X^\star[S])=0$ for all $i
  \leq |S|-r-2$, hence \[\eta(X^\star[S])\geq |S|-r \]
Since $T = V - S \in X$, $T$ does not contain a positive circuit of
$\cal M$, so by hypothesis $\rho(S) = \rho(V - T) \geq r$. Hence
\begin{eqnarray*}
  \eta(Z[\cup_{i\in I}V_i])& \geq &\eta(X^\star[S]) + \rho(S) \\
                         & \geq & (|S| - r) + r = |I|
\end{eqnarray*}
\end{enumerate}
Proposition \ref{meshulam} therefore implies that $Z$ contains a colorful
simplex. $\qed$

\section{Acknowledgments}
The author would like to thank Imre B\'{a}r\'{a}ny and Hyung Seuk Oh
for their helpful comments and valuable time.

\end{document}